\documentclass[12pt,fleqn]{article}

\usepackage{amsmath,amssymb,amsthm}
\usepackage{geometry} 
\usepackage[pdfpagemode=UseNone,pdfstartview=FitH,hypertex]{hyperref}

\newcommand{\abs}[1]{\left|#1\right|}
\newcommand{\bdry}[1]{\partial #1}
\newcommand{\dint}{\ds{\int}}
\newcommand{\ds}[1]{\displaystyle #1}
\newcommand{\eps}{\varepsilon}
\newcommand{\F}{{\cal F}}
\newcommand{\hquad}{\hspace{0.08in}}
\newcommand{\norm}[2][]{\left\|#2\right\|_{#1}}
\renewcommand{\o}{\text{o}}
\newcommand{\PS}[1]{$(\text{PS})_{#1}$}
\newcommand{\R}{\mathbb R}
\newcommand{\RP}{\R \text{P}}
\newcommand{\seq}[1]{\left(#1\right)}
\newcommand{\set}[1]{\left\{#1\right\}}
\newcommand{\wto}{\rightharpoonup}
\newcommand{\Z}{\mathbb Z}

\DeclareMathOperator{\divg}{div}

\newenvironment{enumroman}{\begin{enumerate}

}{\end{enumerate}}

\newtheorem{corollary}{Corollary}[section]
\newtheorem{proposition}[corollary]{Proposition}
\newtheorem{theorem}[corollary]{Theorem}

\theoremstyle{definition}
\newtheorem{definition}[corollary]{Definition}

\theoremstyle{remark}
\newtheorem{remark}[corollary]{Remark}

\numberwithin{equation}{section}

\title{\bf On the critical $p$-Kirchhoff equation}
\author{\bf Erisa Hasani and Kanishka Perera\\
Department of Mathematical Sciences\\
Florida Institute of Technology\\
Melbourne, FL 32901, USA\\
\em ehasani2016@my.fit.edu \& kperera@fit.edu}
\date{}

\begin{document}

\maketitle

\begin{abstract}
We study a nonlocal elliptic equation of $p$-Kirchhoff type involving the critical Sobolev exponent. First we give sufficient conditions for the \PS{} condition to hold. Then we prove some existence and multiplicity results using tools from Morse theory, in particular, the notion of a cohomological local splitting and eigenvalues based on the Fadell-Rabinowitz cohomological index.
\end{abstract}

\section{Introduction and statement of results}

Nonlocal elliptic equations of $p$-Kirchhoff type involving critical Sobolev exponents have been recently studied in the literature (see, e.g., Hamydy et al.\! \cite{MR2832280}, Ourraoui \cite{MR3186916}, Zhou and Song \cite{MR3430291}, Li et al.\! \cite{MR3503048}, Li et al.\! \cite{MR3651881}, and the references therein). In this paper we study the existence and multiplicity of solutions to the critical $p$-Kirchhoff equation
\begin{equation} \label{1}
\left\{\begin{aligned}
- h\bigg(\int_\Omega |\nabla u|^p\, dx\bigg)\, \Delta_p\, u & = f(x,u) + |u|^{p^\ast - 2}\, u && \text{in } \Omega\\[10pt]
u & = 0 && \text{on } \bdry{\Omega},
\end{aligned}\right.
\end{equation}
where $\Omega$ is a bounded domain in $\R^N,\, N \ge 2$, $\Delta_p\, u = \divg (|\nabla u|^{p-2}\, \nabla u)$ is the $p$-Laplacian of $u$, $1 < p < N$, $h : [0,\infty) \to [0,\infty)$ is a continuous and nondecreasing function, $p^\ast = Np/(N - p)$ is the critical Sobolev exponent, and $f$ is a Carath\'{e}odory function on $\Omega \times \R$ satisfying the subcritical growth condition
\begin{equation} \label{2}
|f(x,t)| \le a_1 |t|^{q-1} + a_2 \quad \text{for a.a.\! } x \in \Omega \text{ and all } t \in \R
\end{equation}
for some constants $a_1, a_2 > 0$ and $1 < q < p^\ast$. A model case is
\begin{equation} \label{9}
h(t) = 1 + bt^{\gamma - 1},
\end{equation}
where $b > 0$ and $\gamma > 1$.

Weak solutions of problem \eqref{1} coincide with critical points of the $C^1$-functional
\begin{equation} \label{20}
E(u) = \frac{1}{p}\, H\bigg(\int_\Omega |\nabla u|^p\, dx\bigg) - \int_\Omega F(x,u)\, dx - \frac{1}{p^\ast} \int_\Omega |u|^{p^\ast} dx, \quad u \in W^{1,\,p}_0(\Omega),
\end{equation}
where
\[
H(t) = \int_0^t h(s)\, ds, \qquad F(x,t) = \int_0^t f(x,s)\, ds
\]
are the primitives of $h$ and $f$, respectively. As is typical with problems of critical growth, the main difficulty here is the lack of compactness. Recall that the functional $E$ satisfies the Palais-Smale compactness condition, or the \PS{} condition for short, if every sequence $\seq{u_j} \subset W^{1,\,p}_0(\Omega)$ such that $E(u_j)$ is bounded and $E'(u_j) \to 0$, called a \PS{} sequence, has a convergent subsequence. First we give a sufficient condition for every bounded \PS{} sequence to have a convergent subsequence. Let
\begin{equation} \label{3}
S_{N,\,p} = \inf_{u \in W^{1,\,p}_0(\Omega) \setminus \set{0}}\, \frac{\dint_\Omega |\nabla u|^p\, dx}{\left(\dint_\Omega |u|^{p^\ast} dx\right)^{p/p^\ast}}
\end{equation}
be the best Sobolev constant.

\begin{theorem} \label{Theorem 1}
If \eqref{2} holds and
\begin{equation} \label{4}
h(t) > S_{N,\,p}^{- p^\ast\!/p}\, t^{p^\ast\!/p - 1} \quad \forall t > 0,
\end{equation}
then every bounded sequence $\seq{u_j} \subset W^{1,\,p}_0(\Omega)$ such that $E'(u_j) \to 0$ has a convergent subsequence.
\end{theorem}

In the model case \eqref{9}, the inequality \eqref{4} holds in each of the following cases:
\begin{enumroman}
\item $\gamma = p^\ast/p$ and $b \ge S_{N,\,p}^{- p^\ast\!/p}$,
\item $\gamma > p^\ast/p$ and
    \[
    b^{p^\ast\!/p - 1} > \frac{(\gamma - p^\ast/p)^{\gamma - p^\ast\!/p}\, (p^\ast/p - 1)^{p^\ast\!/p - 1}}{(\gamma - 1)^{\gamma - 1}}\, S_{N,\,p}^{- (p^\ast\!/p)(\gamma - 1)}.
    \]
\end{enumroman}

\begin{remark}
The case where $1 < \gamma < p^\ast\!/p$ can be handled using arguments similar to those used in \cite{1078-0947_2021106}, which only considered the semilinear case $p = 2$.
\end{remark}

Next we give sufficient conditions for the existence of a solution to problem \eqref{1}.

\begin{theorem} \label{Theorem 2}
If \eqref{2} and \eqref{4} hold, and
\begin{equation} \label{11}
H(t) \ge \frac{p}{p^\ast}\, S_{N,\,p}^{- p^\ast\!/p}\, t^{p^\ast\!/p} + a_3\, t^r - a_4 \quad \forall t \ge 0
\end{equation}
for some constants $a_3, a_4 > 0$ and $r > q/p$, then problem \eqref{1} has a solution.
\end{theorem}

In particular, we have the following existence result for the model problem
\begin{equation} \label{13}
\left\{\begin{aligned}
- \left[1 + b \left(\int_\Omega |\nabla u|^p\, dx\right)^{\gamma - 1}\right] \Delta_p\, u & = f(x,u) + |u|^{p^\ast - 2}\, u && \text{in } \Omega\\[10pt]
u & = 0 && \text{on } \bdry{\Omega}.
\end{aligned}\right.
\end{equation}

\begin{corollary} \label{Corollary 1}
If \eqref{2} holds, then problem \eqref{13} has a solution in each of the following cases:
\begin{enumroman}
\item $\gamma = p^\ast/p$, $b = S_{N,\,p}^{- p^\ast\!/p}$, and $q < p$,
\item $\gamma = p^\ast/p$ and $b > S_{N,\,p}^{- p^\ast\!/p}$,
\item $\gamma > p^\ast/p$ and
    \[
    b^{p^\ast\!/p - 1} > \frac{(\gamma - p^\ast/p)^{\gamma - p^\ast\!/p}\, (p^\ast/p - 1)^{p^\ast\!/p - 1}}{(\gamma - 1)^{\gamma - 1}}\, S_{N,\,p}^{- (p^\ast\!/p)(\gamma - 1)}.
    \]
\end{enumroman}
\end{corollary}

Now we assume that
\begin{equation} \label{12}
\lim_{t \to 0}\, \frac{H(t)}{t} = 1
\end{equation}
and
\begin{equation} \label{15}
\lim_{t \to 0}\, \frac{pF(x,t)}{|t|^p} = \lambda \quad \text{uniformly a.e.\! in } \Omega.
\end{equation}
Then problem \eqref{1} has the trivial solution $u \equiv 0$, and we seek nontrivial solutions. The location of $\lambda$ with respect to the spectrum of the $p$-Laplacian will play an important role in our results. We recall that the spectrum $\sigma(- \Delta_p)$ consists of those $\lambda \in \R$ for which the eigenvalue problem
\begin{equation} \label{16}
\left\{\begin{aligned}
- \Delta_p\, u & = \lambda\, |u|^{p-2}\, u && \text{in } \Omega\\[10pt]
u & = 0 && \text{on } \bdry{\Omega}
\end{aligned}\right.
\end{equation}
has a nontrivial solution. The first eigenvalue
\begin{equation} \label{17}
\lambda_1 = \inf \sigma(- \Delta_p)
\end{equation}
is positive, simple, and has an associated eigenfunction $\varphi_1$ that is positive in $\Omega$ (see Anane \cite{MR89e:35124} and Lindqvist \cite{MR90h:35088,MR1139483}). Moreover, $\lambda_1$ is isolated in the spectrum, so the second eigenvalue
\begin{equation} \label{18}
\lambda_2 = \inf \sigma(- \Delta_p) \cap (\lambda_1,\infty)
\end{equation}
is well-defined (see Anane and Tsouli \cite{MR97k:35190}). We have the following theorem.

\begin{theorem} \label{Theorem 3}
Assume \eqref{2}, \eqref{4}, \eqref{11}, \eqref{12}, and \eqref{15} with $\lambda \notin \sigma(- \Delta_p)$.
\begin{enumroman}
\item \label{Theorem 3.i} If $\lambda > \lambda_1$, then problem \eqref{1} has a nontrivial solution.
\item \label{Theorem 3.ii} If $\lambda > \lambda_2$, then problem \eqref{1} has two nontrivial solutions.
\end{enumroman}
\end{theorem}

In particular, we have the following corollary for the model problem \eqref{13}.

\begin{corollary}
Assume \eqref{2} and \eqref{15} with $\lambda \notin \sigma(- \Delta_p)$.
\begin{enumroman}
\item If $\lambda > \lambda_1$, then problem \eqref{13} has a nontrivial solution in each of the three cases in Corollary \ref{Corollary 1}.
\item If $\lambda > \lambda_2$, then problem \eqref{13} has two nontrivial solutions in each of the three cases in Corollary \ref{Corollary 1}.
\end{enumroman}
\end{corollary}

The proof of Theorem \ref{Theorem 3} uses tools from Morse theory, in particular, the notion of a cohomological local splitting (see Degiovanni et al.\! \cite{MR2661274}, Perera et al.\! \cite{MR2640827}, and Perera \cite{MR1700283}) and eigenvalues based on the cohomological index (see Perera \cite{MR1998432} and Perera and Szulkin \cite{MR2153141}). We will recall these tools in the next section.

\section{Preliminaries}

Let $E$ be a $C^1$-functional defined on a Banach space $W$. For $a \in \R$, we denote by $E^a$ the sublevel set $\set{u \in W : E(u) \le a}$, and we denote by $H$ the Alexander-Spanier cohomology with $\Z_2$ coefficients (see Spanier \cite{MR1325242}). In Morse theory, the local behavior of $E$ near a critical point $u_0$ is described by the sequence of critical groups
\[
C^q(E,u_0) = H^q(E^c,E^c \setminus \set{u_0}), \quad q \ge 0,
\]
where $c = E(u_0)$ is the corresponding critical value (see Chang \cite{MR1196690}, Mawhin and Willem \cite{MR982267}, and Perera et al.\! \cite{MR2640827}).

Recall that $E$ satisfies the Palais-Smale compactness condition at the level $c \in \R$, or the \PS{c} condition for short, if every sequence $\seq{u_j} \subset W$ such that $E(u_j) \to c$ and $E'(u_j) \to 0$ has a convergent subsequence. The proof of Theorem \ref{Theorem 3} will make use of the following alternative proved in Perera \cite{MR1749421} (see also Perera et al.\! \cite[Proposition 3.28({\em ii})]{MR2640827}).

\begin{proposition} \label{Proposition 1}
Assume that zero is a critical point of $E$ with $E(0) = 0$ and $C^k(E,0) \ne 0$ for some $k \ge 1$, and that there are regular values $- \infty < a < 0 < b \le + \infty$ of $E$ with $H^k(E^b,E^a) = 0$ such that $E$ has only a finite number of critical points in $E^{-1}([a,b])$ and $E$ satisfies the {\em \PS{c}} condition for all $c \in [a,b] \cap \R$. Then $E$ has a nontrivial critical point $u_1$ with either
\[
a < E(u_1) < 0 \hquad \text{and} \hquad C^{k-1}(E,u_1) \ne 0,
\]
or
\[
0 < E(u_1) < b \hquad \text{and} \hquad C^{k+1}(E,u_1) \ne 0.
\]
\end{proposition}

To obtain a nontrivial critical group at zero in the absence of a suitable direct sum decomposition, we will use a cohomological local splitting. For a symmetric set $A \subset W \setminus \set{0}$, let $\overline{A} = A/\Z_2$ be the quotient space of $A$ with each $u$ and $-u$ identified, let $f : \overline{A} \to \RP^\infty$ be the classifying map of $\overline{A}$, and let $f^\ast : H^\ast(\RP^\infty) \to H^\ast(\overline{A})$ be the induced homomorphism of the cohomology rings. The $\Z_2$-cohomological index of $A$ is defined by
\[
i(A) = \begin{cases}
0 & \text{if } A = \emptyset\\[7.5pt]
\sup \set{m \ge 1 : f^\ast(\omega^{m-1}) \ne 0} & \text{if } A \ne \emptyset,
\end{cases}
\]
where $\omega \in H^1(\RP^\infty)$ is the generator of the polynomial ring $H^\ast(\RP^\infty) = \Z_2[\omega]$ (see Fadell and Rabinowitz \cite{MR0478189}).

\begin{definition}
We say that $E$ has a cohomological local splitting near zero in dimension $0 \le k < \infty$ if there are disjoint nonempty closed symmetric subsets $A_0$ and $B_0$ of the unit sphere $S = \set{u \in W : \norm{u} = 1}$ with
\begin{equation} \label{22}
i(A_0) = i(S \setminus B_0) = k
\end{equation}
and $\rho > 0$ such that, setting $A = \set{tu : u \in A_0,\, 0 \le t \le \rho}$ and $B = \set{tu : u \in B_0,\, 0 \le t \le \rho}$, we have
\begin{equation} \label{24}
\sup_A\, E \le 0 \le \inf_B\, E.
\end{equation}
\end{definition}

This definition was given, in an equivalent form, in Degiovanni et al.\! \cite{MR2661274} and is a slight variant of Perera et al.\! \cite[Definition 3.33]{MR2640827}, which in turn is a variant of the homological local linking of Perera \cite{MR1700283}. The following proposition was proved in Degiovanni et al.\! \cite{MR2661274} (see also Perera et al.\! \cite[Proposition 3.34]{MR2640827} and Perera \cite{MR1700283}).

\begin{proposition} \label{Proposition 2}
If zero is an isolated critical point of $E$ and $E$ has a cohomological local splitting near zero in dimension $k$, then $C^k(E,0) \ne 0$.
\end{proposition}

To show that the functional $E$ in \eqref{20} has a cohomological local splitting near zero when $\lambda \notin \sigma(- \Delta_p)$, we will make use of a sequence of eigenvalues based on the cohomological index that was first introduced in Perera \cite{MR1998432} (see also Perera and Szulkin \cite{MR2153141}). Recall that eigenvalues of problem \eqref{16} coincide with critical values of the functional
\[
\Psi(u) = \frac{1}{\dint_\Omega |u|^p\, dx}, \quad u \in S = \set{u \in W^{1,\,p}_0(\Omega) : \int_\Omega |\nabla u|^p\, dx = 1}.
\]
Denote by $\F$ the class of symmetric subsets of $S$, let
\[
\F_k = \set{M \in \F : i(M) \ge k},
\]
and set
\begin{equation} \label{21}
\lambda_k = \inf_{M \in \F_k}\, \sup_{u \in M}\, \Psi(u), \quad k \ge 1.
\end{equation}
Then $\lambda_1$ and $\lambda_2$ agree with \eqref{17} and \eqref{18}, respectively, and $\seq{\lambda_k}$ is a nondecreasing and unbounded sequence of eigenvalues. Moreover, denoting by
\[
\Psi^a = \set{u \in S : \Psi(u) \le a}, \qquad \Psi_a = \set{u \in S : \Psi(u) \ge a}
\]
the sublevel and superlevel sets of $\Psi$, respectively, we have
\begin{equation} \label{19}
\lambda_k < \lambda_{k+1} \implies i(\Psi^{\lambda_k}) = i(S \setminus \Psi_{\lambda_{k+1}}) = k
\end{equation}
(see Perera et al.\! \cite[Theorem 4.6]{MR2640827}). In the next section, we will make use of \eqref{19} to show that if $\lambda_k < \lambda < \lambda_{k+1}$, then $E$ has a cohomological local splitting near zero in dimension $k$.

\section{Proofs}

In this section we prove Theorem \ref{Theorem 1}, Theorem \ref{Theorem 2}, and Theorem \ref{Theorem 3}.

\begin{proof}[Proof of Theorem \ref{Theorem 1}]
Since $\seq{u_j}$ is bounded, for a renamed subsequence,
\begin{equation} \label{5}
u_j \wto u, \qquad \norm{u_j - u}^p \to t
\end{equation}
for some $u \in W^{1,\,p}_0(\Omega)$ and $t \ge 0$. Since $E'(u_j) \to 0$,
\begin{multline} \label{6}
\hspace{-1.2pt} h\bigg(\int_\Omega |\nabla u_j|^p\, dx\bigg) \int_\Omega |\nabla u_j|^{p-2}\, \nabla u_j \cdot \nabla v\, dx - \int_\Omega f(x,u_j)\, v\, dx - \int_\Omega |u_j|^{p^\ast - 2}\, u_j\, v\, dx = \o(\norm{v})\\[10pt]
\forall v \in W^{1,\,p}_0(\Omega).
\end{multline}
By the Br{\'e}zis-Lieb lemma (see \cite{MR699419}) and \eqref{5},
\[
\int_\Omega |\nabla u_j|^p\, dx \to \int_\Omega |\nabla u|^p\, dx + t =: s.
\]
For a further subsequence, $u_j \to u$ strongly in $L^q(\Omega)$ and a.e.\! in $\Omega$. So taking $v = u_j$ in \eqref{6} gives
\begin{equation} \label{7}
h(s) \left(\int_\Omega |\nabla u|^p\, dx + t\right) - \int_\Omega uf(x,u)\, dx - \int_\Omega |u_j|^{p^\ast} dx = \o(1),
\end{equation}
while taking $v = u$ and passing to the limit gives
\begin{equation} \label{8}
h(s) \int_\Omega |\nabla u|^p\, dx - \int_\Omega uf(x,u)\, dx - \int_\Omega |u|^{p^\ast} dx = 0.
\end{equation}
Since
\[
\int_\Omega |u_j|^{p^\ast} dx - \int_\Omega |u|^{p^\ast} dx = \int_\Omega |u_j - u|^{p^\ast} dx + \o(1)
\]
by the Br{\'e}zis-Lieb lemma, subtracting \eqref{8} from \eqref{7} and using \eqref{3} gives
\[
th(s) = \int_\Omega |u_j - u|^{p^\ast} dx + \o(1) \le S_{N,\,p}^{- p^\ast\!/p} \left(\int_\Omega |\nabla (u_j - u)|^p\, dx\right)^{p^\ast\!/p} + \o(1).
\]
Noting that $h(s) \ge h(t)$ since $s \ge t$ and $h$ is nondecreasing, and passing to the limit gives $th(t) \le S_{N,\,p}^{- p^\ast\!/p}\, t^{p^\ast\!/p}$, which together with \eqref{4} implies that $t = 0$. So $u_j \to u$.
\end{proof}

\begin{proof}[Proof of Theorem \ref{Theorem 2}]
The inequality \eqref{11} together with \eqref{2} and \eqref{3} gives
\[
E(u) \ge a_5 \left(\int_\Omega |\nabla u|^p\, dx\right)^r - a_6 \int_\Omega |u|^q\, dx - a_7
\]
for some constants $a_5, a_6, a_7 > 0$. Since $q < p^\ast$ and $r > q/p$, this together with the Sobolev embedding implies that $E$ is bounded from below and coercive. Coercivity implies that every \PS{} sequence is bounded, so $E$ satisfies the \PS{} condition by Theorem \ref{Theorem 1}. So $E$ has a global minimizer.
\end{proof}

\begin{proof}[Proof of Theorem \ref{Theorem 3}]
As in the proof of Theorem \ref{Theorem 2}, $E$ is bounded from below and has a global minimizer $u_0$. We may assume without loss of generality that $E$ has only a finite number of critical points and hence all critical points of $E$ are isolated. Then the critical groups of $E$ at $u_0$ are given by
\begin{equation} \label{25}
C^q(E,u_0) = \begin{cases}
\Z_2 & \text{if } q = 0\\[7.5pt]
0 & \text{otherwise}.
\end{cases}
\end{equation}

Next we show that $E$ has a cohomological local splitting near zero when $\lambda > \lambda_1$. Let $\seq{\lambda_k}$ be the sequence of eigenvalues in \eqref{21}. Since $\lambda \notin \sigma(- \Delta_p)$, we have $\lambda_k < \lambda < \lambda_{k+1}$ for some $k \ge 1$. By \eqref{19},
\[
i(\Psi^{\lambda_k}) = i(S \setminus \Psi_{\lambda_{k+1}}) = k.
\]
Let
\[
A_0 = \Psi^{\lambda_k}, \qquad B_0 = \Psi_{\lambda_{k+1}}.
\]
Then $A_0$ and $B_0$ are disjoint nonempty closed symmetric subsets of $S$ satisfying \eqref{22}. Fix $\eps > 0$ so small that $\lambda - \eps > \lambda_k$ and $\lambda + \eps < \lambda_{k+1}$. By \eqref{15} and \eqref{2}, $\exists\, a_\eps > 0$ such that
\begin{equation} \label{23}
\abs{F(x,t) - \frac{\lambda}{p}\, |t|^p} \le \frac{\eps}{p}\, |t|^p + a_\eps |t|^{p^\ast} \quad \text{for a.a.\! } x \in \Omega \text{ and all } t \in \R.
\end{equation}
For $u \in S$ and $t \ge 0$,
\[
H\bigg(\int_\Omega |\nabla (tu)|^p\, dx\bigg) = H(t^p) = t^p + \o(t^p) \text{ as } t \to 0
\]
by \eqref{12}, and
\[
\abs{\int_\Omega F(x,tu)\, dx - \frac{\lambda t^p}{p} \int_\Omega |u|^p\, dx} \le \frac{\eps t^p}{p} \int_\Omega |u|^p\, dx + \o(t^p) \text{ as } t \to 0
\]
by \eqref{23} and the Sobolev embedding, so
\[
\frac{t^p}{p} \left(1 - \frac{\lambda + \eps}{\Psi(u)}\right) + \o(t^p) \le E(tu) \le \frac{t^p}{p} \left(1 - \frac{\lambda - \eps}{\Psi(u)}\right) + \o(t^p) \text{ as } t \to 0.
\]
In particular,
\begin{equation} \label{26}
E(tu) \le - \frac{t^p}{p} \left(\frac{\lambda - \eps}{\lambda_k} - 1\right) + \o(t^p)
\end{equation}
for $u \in A_0$, and
\[
E(tu) \ge \frac{t^p}{p} \left(1 - \frac{\lambda + \eps}{\lambda_{k+1}}\right) + \o(t^p)
\]
for $u \in B_0$, so \eqref{24} holds for sufficiently small $\rho > 0$. So $E$ has a cohomological local splitting near zero in dimension $k$.

\ref{Theorem 3.i} By \eqref{26}, $E(tu) < 0$ for $u \in A_0$ and $t > 0$ sufficiently small, so $E(u_0) = \inf E < 0$ and hence $u_0$ is nontrivial.

\ref{Theorem 3.ii} Since $E$ has a cohomological local splitting near zero in dimension $k$, $C^k(E,0) \ne 0$ by Proposition \ref{Proposition 2}. For $a < \inf E$ and $b = + \infty$,
\[
H^k(E^b,E^a) = H^k(W^{1,\,p}_0(\Omega)) = 0
\]
since $k \ge 1$, so $E$ has a nontrivial critical point $u_1$ with either $E(u_1) < 0$ and $C^{k-1}(E,u_1) \ne 0$, or $E(u_1) > 0$ and $C^{k+1}(E,u_1) \ne 0$ by Proposition \ref{Proposition 1}. Since $\lambda > \lambda_2$, $k \ge 2$ and hence $C^{k-1}(E,u_0) = 0 = C^{k+1}(E,u_0)$ by \eqref{25}, so $u_1$ is distinct from $u_0$.
\end{proof}

\def\cdprime{$''$}

\end{document}